\documentclass[12pt]{amsart}
\usepackage{amsmath,amsthm,amssymb,bbm,dsfont}
\usepackage{graphicx}
\usepackage{subfig}
\usepackage{cite,float}
\usepackage{hyperref}

\numberwithin{equation}{section}
\newtheorem{theorem}{Theorem}[section]
\newtheorem{lemma}[theorem]{Lemma}

\theoremstyle{remark}
\newtheorem{remark}[theorem]{Remark}

\makeatletter
\@namedef{subjclassname@2020}{%
  \textup{2020} Mathematics Subject Classification}
\makeatother
\begin{document}
\sloppy
\title [On Fridman invariant, injectivity radius function and...]{On Fridman invariant, injectivity radius function and squeezing function}

\author[A. Kumar]{Akhil Kumar}
\address{Department of Mathematics, University of Delhi,
Delhi--110 007, India}
\email{98.kumarakhil@gmail.com}

\author[S. K. Pant]{Sanjay Kumar Pant}
\address{Department of Mathematics, Deen Dayal Upadhyaya College, University of Delhi,
Delhi--110 078, India}
\email{skpant@ddu.du.ac.in}

\begin{abstract}
We give a class of domains for which Fridman invariant and injectivity radius function coincide with respect to Carathéodory metric. We give explicit expressions of the squeezing functions for these  domains and investigate some of their properties.  
\end{abstract}
\keywords{Squeezing function;  Injectivity radius function; Fridman invariants.}
\subjclass[2020]{32F45, 32H02}
\maketitle

\section{Introduction}
In \cite{recent2}, Ng, Tang and Tsai raised a  question: ``For which $c$-hyperbolic complex manifold $X$ we have $H_{X}^{c}(z) \leq i_{X}^{c}(z)$ for all $z\in X$?" In one complex dimension, they have already shown that $H_{X}^{k}(z) = i_{X}^{k}(z)$ for all $z\in X$, where $X$ is any 1-dimensional $k$-hyperbolic complex manifold. In the context of this question, we state our main result
\begin{theorem}\rm\label{main1}
Let  $\Omega = \mathbb{D}\setminus K$ be a domain in $\mathbb{C}$, where $\mathbb{D} = \lbrace \xi\in \mathbb{C}:  |\xi| < 1 \rbrace$ and $K$ is a removable set in $\mathbb{D}$. Then, we have 
$$H_{\Omega}^{c}(z) = i_{\Omega}^{c}(z) = \min_{w\in K}\left\lbrace \left|\dfrac{w-z}{1-\overline{z}w}\right|\right\rbrace \  \mbox{for all}\ z\in \Omega.$$
\end{theorem} 
 
Here, $H_{\Omega}^{c}(z)$ and $i_{\Omega}^{c}(z)$ refer to the Fridman invariant and the injectivity radius function with respect to the Carathéodory metric. 

In \cite{frid1979, frid1983}, Fridman introduced a holomorphic invariant known as the Fridman invariant. Let $X$ be an $n$-dimensional $d$-hyperbolic complex manifold. For $z\in X$, the Fridman invariant denoted by $h_{X}^{d}$, is defined as
$$h_{X}^{d}(z)= \inf\left\lbrace \dfrac{1}{r} : B_{X}^{d}(z, r)\subseteq f(\mathbb{B}^{n}), f:\mathbb{B}^{n}\to X \right\rbrace,$$
where $f:\mathbb{B}^{n}\to X$ is an injective holomorphic map and $B_{X}^{d}(z, r)$ is a ball centred at $z$ with radius $r$ with respect to $d$-metric. The Fridman invariant  is very useful in decoding some geometrical properties of the  manifold. For example, Fridman \cite{frid1983} has shown that if Fridman invariant with respect to Kobayashi metric $h_{X}^{k}(z) = 0$ for some $z\in X$, where $X$ is any connected $k$-hyperbolic complex manifold, then $X$ is biholomorphic to $\mathbb{B}^{n}$. For comparison purposes, Nikolov and Verma \cite{kaushal} considered a modification in $h_{X}^{d}$, which is denoted by $H_{X}^{d}$ and defined as 
$$H_{X}^{d}(z) = \sup \lbrace \tanh r : B_{X}^{d}(z, r) \subseteq f(\mathbb{B}^{n}), f:\mathbb{B}^{n}\to X \rbrace,$$
where $f:\mathbb{B}^{n}\to X$ is an injective holomorphic map. 

Let $X$ be an $n$-dimensional $d$-hyperbolic complex manifold. For $z\in X$, the injectivity radius function with respect to $d$-metric denoted by $i_{X}^{d}$, is defined as
$$i_{X}^{d}(z) = \sup \lbrace \tanh r : B_{X}^{d}(z, r)\ \mbox{is simply connected} \rbrace.$$
For Carathéodory and Kobayashi metrics, we denote injectivity radius functions by $i_{X}^{c}$ and $i_{X}^{k}$ respectively. The injectivity radius with respect to $d$-metric is defined as
$$i_{X}^{d} = \inf _{z\in X} i_{X}^{d}(z).$$
A Riemann Surface $X$ is said to be uniformly thick if the injectivity radius function $i_{X}^{k}(z)$ on it admits a positive lower bound. In \cite{recent2}, Ng, Tang and Tsai extended the definition of uniformly thickness to higher dimension. An $n$-dimensional $d$-hyperbolic complex manifold is said to be $d$-uniformly thick if Fridman invariant $H_{X}^{d}(z)$ on it, admits a positive lower bound. For example, all simply connected bounded domains in $\mathbb{C}$ are both $c$-uniformly thick and $k$-uniformly thick. For other important properties of injectivity radius function, one can refer to the following papers: \cite{injectivity1, injectivity3, injectivity2}.

Let us recall the definition of a removable set. A compact set $K$ in $\mathbb{C}$ is said to be removable for bounded holomorphic functions, or simply removable, if for every open superset $U$ of $K$ in $\mathbb{C}$, each bounded and holomorphic function on $U\setminus K$ extends holomorphically on the whole of $U$. One can see that a removable set $K$ in $\mathbb{C}$ is totally disconnected. 

\begin{remark}\rm\label{main300}
Theorem \ref{main1} implies that for the class of domains $\Omega = \Omega'\setminus K$, where $\Omega'$ is any simply connected domain other than $\mathbb{C}$ and $K$ is any removable set in $\Omega'$, the Fridman invariant and the injectivity radius function coincide with respect to the Carathéodory metric.
\end{remark}

\begin{remark}\rm\label{main301}
In Theorem \ref{main1}, we take the domain obtained by deleting a removable set from the unit disc. It is natural to ask that what happens to the Fridman invariant and the injectivity radius function if we replace the unit disc by an arbitrary domain in Theorem \ref{main1}? In Theorem \ref{main10}, we show that one side of the inequality holds.
\end{remark}

\begin{theorem}\rm\label{main10}
Let  $\Omega$ be a domain in $\mathbb{C}$ and $K$ be a removable set in $\Omega$. Then for $D = \Omega\setminus K$, we have
$$i_{D}^{c}(z) \leq H_{D}^{c}(z) \leq \min_{w\in K} \tanh[c_{\Omega}(z, w)].$$
\end{theorem}

\begin{remark}\rm\label{main320}
Theorem \ref{main10} implies that the domain $D = \Omega\setminus K$, where $\Omega$ is any domain in $\mathbb{C}$ and $K$ is any removable set in $\Omega$, is not $c$-uniformly thick.  Theorem \ref{main11} is about a domain  where the  above inequality does not hold in general. 
\end{remark}

\begin{theorem}\rm\label{main11}
Let  $D = A_{r}\setminus \lbrace p\rbrace$ be a domain in $\mathbb{C}$, where $A_{r} = \lbrace \xi \in \mathbb{C} : r < |\xi| < 1, r > 0 \rbrace$ and $p \in (-1,-\sqrt{r})$. Then for $z\in (r, 1)$
$$i_{D}^{c}(z) \neq  \tanh[c_{A_{r}}(z, p)]\   \mbox{and}\ H_{D}^{c}(z) \neq  \tanh[c_{A_{r}}(z, p)].$$

\end{theorem}

Our next result gives the explicit expression of   the squeezing function  for the domain given in the Theorem  \ref{main1}.
\begin{theorem}\rm\label{main2}
Let $\Omega = \mathbb{D}\setminus K$ be an open connected set. Then the squeezing function on $\Omega$ is given by
$$S_{\Omega}(z) = \min_{w\in K}\left\lbrace \left|\dfrac{w-z}{1-\overline{z}w}\right|\right\rbrace.$$
\end{theorem}
The  $S_{\Omega}(z)$ refer squeezing function which is defined as follows. In 2012, Deng, Guan and Zhang in their paper \cite{2012}, developed the idea of squeezing function by following the work of Liu, Sun and Yau \cite{Yau2004}, \cite{Yau2005} and Yeung \cite{yeung}. 

Let $\Omega\subseteq \mathbb{C}^n$ be a bounded domain. For $z\in{\Omega}$ and a injective holomorphic map $f:\Omega\to \mathbb{B}^n$  with $f(z)=0$, define 
$$S_{\Omega}(z,f) = \sup \{r:\mathbb{B}^n(0,r)\subseteq f(\Omega)\},$$
where $\mathbb{B}^n$ denotes the unit ball in $\mathbb{C}^n$ and $\mathbb{B}^n(0,r)$ denotes the ball entered at origin with radius $r$ in $\mathbb{C}^n$.
The squeezing function on $\Omega$, denoted by $S_{\Omega}$, is defined as
$$S_{\Omega}(z) = \sup_f\{S_{\Omega}(z,f)\},$$ 
where supremum is taken over all injective holomorphic maps $f:\Omega\to \mathbb{B}^n$ with $f(z)=0.$

It is clear that the Fridman invariant, the injectivity radius function and the squeezing function are biholomorphic invariants. It is interesting to note that any domain in class $\Omega'\setminus K$, where $\Omega'$ is any simply connected domain other than $\mathbb{C}$ and $K$ is any removable set in $\Omega'$, is not a holomorphic homogeneous regular domain. A domain is said to be a holomorphic homogeneous regular (HHR) domain if the squeezing function on it admits a positive lower bound. For other important properties of squeezing functions, see, for example: \cite{deng2019, uniform-squeezing, kaushalprachi, recent, rong, rong2, rong3} and references therein.

\begin{remark}\rm\label{main302}
Now the same question arises for the squeezing function: what about the squeezing function if we replace an arbitrary domain instead of the unit disc in Theorem \ref{main2}? In Theorem \ref{main3}, we see that the one-sided inequality still holds.
\end{remark}

\begin{theorem}\rm\label{main3}
Let  $\Omega$ be a bounded domain in $\mathbb{C}$ and $K$ be a removable set in $\Omega$. Then for $D = \Omega\setminus K$, we have
$$S_{D}(z) \leq \min_{w\in K} \tanh[c_{\Omega}(z, w)].$$
\end{theorem}

\begin{remark}\rm\label{main320}
Theorem \ref{main3} implies that the domain $D = \Omega\setminus K$, where $\Omega$ is any bounded domain in $\mathbb{C}$ and $K$ is any removable set in $\Omega$, is not a HHR domain. In Theorem \ref{main15}, we give an example of domain to illustrate that the above inequality does not hold in general.
\end{remark}

\begin{theorem}\rm\label{main15}
Let  $D = A_{r}\setminus \lbrace -\sqrt{r} \rbrace$ be a domain, where $A_{r} = \lbrace \xi \in \mathbb{C} : r < |\xi| < 1, r > 0 \rbrace$. Then for $z\in(r, \sqrt{r})$ 
$$S_{D}(z) \neq \tanh[c_{A_{r}}(z, -\sqrt{r})].$$
\end{theorem}

\section{Some notations and Observations} 
The Carathéodory pseudo distance between $\xi_{1}$ and $\xi_{2}$ on a domain $\Omega\subseteq \mathbb{C}^{n}$, denoted by $c_{\Omega}(\xi_{1}, \xi_{2})$, is defined as 
$$c_{\Omega}(\xi_{1}, \xi_{2}) = \displaystyle\sup_{f}\lbrace \tanh^{-1}|\alpha| : f:\Omega\to \mathbb{D}\ \text{holomorphic}, f(\xi_{1})=0, f(\xi_{2})= \alpha \rbrace.$$ 
Let $d_{\rho}(\xi_{1}, \xi_{2})$ denote the Poincaré distance between $\xi_{1}$ and $\xi_{2}$ in $\mathbb{D}$. Let $\mu:[0, 1)\to \mathbb{R}_{\geq 0}$ be a real-valued function defined as
$\mu(x) = \dfrac{1}{2}\log\dfrac{1+x}{1-x}$. It is clear that $\mu$ is a strictly increasing function, and for $\xi\in\mathbb{D}$, $\mu(|\xi|)$ is the Poincaré distance between 0 and $\xi$. One can observe that the inverse of this function can be expressed as $\tanh(x)$. 

In \cite{cara}, Simha given the explicit formula for Carathéodory distance for the annulus $D = \lbrace \xi\in \mathbb{C} : 1/R < |\xi| < R \rbrace $. Let $\xi_{1}, \xi_{2}\in D$, then 
\begin{equation}\label{eq400:sub}
c_{D}(\xi_{1}, \xi_{2}) = \tanh^{-1}c'_{D}(\xi_{1}, \xi_{2}),
\end{equation}
where
\begin{equation}\label{eq500:sub}
c'_{D}(\xi_{1}, \xi_{2}) = \dfrac{1}{R|\xi_{2}|}|f(\xi_{1}, \xi_{2})|\cdot|f(|\xi_{1}|^{-1}, -|\xi_{2}|)|,
\end{equation}
and for $\xi_{1} > 0$,
$$f(\xi_{1}, \xi_{2}) = \left(1-\dfrac{\xi_{2}}{\xi_{1}}\right)\dfrac{\displaystyle\prod_{n=1}^{\infty}\left(1-\dfrac{\xi_{2}}{R^{4n}\xi_{1}}\right) \left(1-\dfrac{\xi_{1}}{R^{4n}\xi_{2}}\right)}{\displaystyle\prod_{n=1}^{\infty}\left(1-\dfrac{\xi_{1}\xi_{2}}{R^{4n-2}}\right)\left(1-\dfrac{1}{R^{4n-2}\xi_{1}\xi_{2}}\right)}.$$

After some computation, the expression can be reformulated in terms of the Schottky-Klein prime function $\omega(\xi_{1}, \xi_{2})$ for annulus. Let $\xi_{1} > 0, \xi_{2}\in A_{r} = \lbrace \xi\in \mathbb{C} : r < |\xi| < 1, r > 0 \rbrace $, then the Carathéodory distance between $\xi_{1}$ and $\xi_{2}$ on $A_{r}$ is given by
\begin{equation}\label{eq401:sub}
c_{A_{r}}(\xi_{1}, \xi_{2}) = \tanh^{-1}c'_{A_{r}}(\xi_{1}, \xi_{2}),
\end{equation}
where
\begin{equation}\label{eq501:sub}
c'_{A_{r}}(\xi_{1}, \xi_{2}) = \dfrac{1}{\xi_{1}}\left| \dfrac{\omega(\xi_{2},\xi_{1})}{\xi_{1}\omega(\xi_{2},\xi_{1}^{-1})}\right|\left| \dfrac{\omega\left(\xi_{1}, -\dfrac{r}{|\xi_{2}|}\right)}{\dfrac{r}{|\xi_{2}|}\omega\left(\xi_{1}, -\dfrac{|\xi_{2}|}{r}\right)}\right|,
\end{equation}
and $\omega(\xi_{1}, \xi_{2})$ is expressed as
$$\omega(\xi_{1}, \xi_{2}) = (\xi_{1}-\xi_{2})\displaystyle\prod_{n=1}^{\infty}\dfrac{(\xi_{1}-r^{2n}\xi_{2})(\xi_{2}-r^{2n}\xi_{1})}{(\xi_{1}-r^{2n}\xi_{1})(\xi_{2}-r^{2n}\xi_{2})}\   \mbox{for}\  \xi_{1},\xi_{2}\in \mathbb{C}\setminus \lbrace 0 \rbrace.$$
The Schottky-Klein prime function satisfies the following identities see {\cite[Section~5.3]{crowdy}}.
\begin{equation}\label{eq1001:sub}
\omega(\xi_{1}, \xi_{2}) = -\omega(\xi_{2}, \xi_{1})
\end{equation}
\begin{equation}\label{eq1002:sub}
\omega(\overline{\xi_{1}}, \overline{\xi_{2}}) = \overline{\omega(\xi_{1}, \xi_{2})}
\end{equation}
\begin{equation}\label{eq1003:sub}
\omega(1/\xi_{1}, 1/\xi_{2}) = -\dfrac{\omega(\xi_{1}, \xi_{2})}{\xi_{1}\xi_{2}}
\end{equation}
\begin{equation}\label{eq1004:sub}
\omega(r^{2}\xi_{1}, \xi_{2}) = -\dfrac{\xi_{2}\omega(\xi_{1}, \xi_{2})}{\xi_{1}}
\end{equation}

\begin{lemma}\rm\label{main500}
Fix $z > 0 \in A_{r}$. Then we have $c_{A_{r}}(z, \overline{w}) =  c_{A_{r}}(z, w)\  \mbox{for all}\  w\in A_{r}.$
\end{lemma}
\begin{proof}
Since by \eqref{eq501:sub} 
\begin{align*}
c'_{A_{r}}(z, \overline{w}) = \dfrac{1}{z}\left| \dfrac{\omega(\overline{w}, z)}{z\omega(\overline{w}, z^{-1})}\right|\left| \dfrac{\omega\left(z, -\dfrac{r}{|\overline{w}|}\right)}{\dfrac{r}{|\overline{w}|}\omega\left(z, -\dfrac{|\overline{w}|}{r}\right)}\right| = \dfrac{1}{z}\left| \dfrac{\omega(\overline{w}, \overline{z})}{z\omega(\overline{w}, \overline{z}^{-1})}\right|\left| \dfrac{\omega\left(z, -\dfrac{r}{|w|}\right)}{\dfrac{r}{|w|}\omega\left(z, -\dfrac{|w|}{r}\right)}\right|.
\end{align*}
By \eqref{eq1002:sub} 
\begin{align*}
c'_{A_{r}}(z, \overline{w}) = \dfrac{1}{z}\left| \dfrac{\overline{\omega(w, z)}}{z\overline{\omega(w, z^{-1})}}\right|\left| \dfrac{\omega\left(z, -\dfrac{r}{|w|}\right)}{\dfrac{r}{|w|}\omega\left(z, -\dfrac{|w|}{r}\right)}\right| = c'_{A_{r}}(z, w)
\end{align*} 
It follows that $c_{A_{r}}(z, \overline{w}) =  c_{A_{r}}(z, w)\  \mbox{for all}\  w\in A_{r}.$
\end{proof}

\begin{lemma}{\cite[Lemma 1]{recent2}}\rm\label{main200}
Let $r < z < 1$ and $-1 < w < -r$. Then we have the following.
\begin{enumerate}
\item $c_{A_{r}}(z, w) = c_{A_{r}}\left(z, \dfrac{r}{w}\right)$.

\item Fix $z$ and let $c_{A_{r}}(z, w)$ be a function of $w$, where $w$ varies between $(-1, -r)$. Then the minimum value of $c_{A_{r}}(z, w)$ is attained at $w = -\sqrt{r}$ i.e.
$$c_{A_{r}}(z, -\sqrt{r}) < c_{A_{r}}(z, w)\  \mbox{for all}\  w\in (-1, -r)\setminus \lbrace -\sqrt{r} \rbrace.$$
\end{enumerate}

\end{lemma}

\section{Proof of theorems}

\begin{proof}[Proof of Theorem \ref{main1}]
Let us assume $ p = \displaystyle\min_{w\in K}c_{\mathbb{D}}(z, w)$. First we prove that the ball $B_{\Omega}^{c}(z, p) = \lbrace \xi\in \Omega : c_{\Omega}(z, \xi) < p\rbrace$ is simply connected. Since $K$ is removable set, it is easy to see that $c_{\Omega}(\xi_{1}, \xi_{2}) = c_{\mathbb{D}}(\xi_{1}, \xi_{2})$ for all $\xi_{1}, \xi_{2}\in \Omega$. This implies $B_{\Omega}^{c}(z, p) = \lbrace \xi\in \Omega : c_{\mathbb{D}}(z, \xi) < p\rbrace$. We claim that $B_{\Omega}^{c}(z, p) = B_{\mathbb{D}}^{c}(z, p)$.  It is obvious that $B_{\Omega}^{c}(z, p) \subseteq B_{\mathbb{D}}^{c}(z, p)$. For other side containment, let us assume there is some $\xi'\in B_{\mathbb{D}}^{c}(z, p)$ with $\xi'\notin B_{\Omega}^{c}(z, p)$. This implies $\xi'\in K$ such that $c_{\mathbb{D}}(z, \xi') < p$. This contradict the fact that $ p = \displaystyle\min_{w\in K}c_{\mathbb{D}}(z, w)$. Therefore, $B_{\Omega}^{c}(z, p) = B_{\mathbb{D}}^{c}(z, r)$ which is just a Poincaré ball in the unit disc. Hence $B_{\Omega}^{c}(z, p)$ is simply connected. Thus 
\begin{equation}\label{eq5:sub}
i_{\Omega}^{c}(z) \geq \tanh p = \min_{w\in K} \left\lbrace \left|\dfrac{w-z}{1-\overline{z}w}\right|\right\rbrace.
\end{equation}
For $z\in \Omega$, let $f: \mathbb{D}\to \Omega$ be any injective holomorphic function with $B_{\Omega}^{c}(z, r)\subseteq f(\mathbb{D}) \subseteq \Omega$. This implies
\begin{equation}\label{eq6:sub}
B_{\Omega}^{c}(z, r) = \lbrace \xi\in \Omega : c_{\mathbb{D}}(z, \xi) < r\rbrace \subseteq f(\mathbb{D}) \subseteq \Omega.
\end{equation}
Since $f(\mathbb{D})$ is simply connected domain in $\Omega$, we have $B_{\Omega}^{c}(z, r) = B_{\mathbb{D}}^{c}(z, r)$. By \eqref{eq6:sub} 
\begin{equation}\label{eq7:sub}
B_{\mathbb{D}}^{c}(z, r) = \lbrace \xi\in \mathbb{D} : c_{\mathbb{D}}(z, \xi) < r\rbrace\subseteq f(\mathbb{D}) \subseteq \Omega.
\end{equation}
Consider a conformal map $g : \mathbb{D}\to \mathbb{D}$ such that 
$$g(\xi) = \dfrac{\xi-z}{1-\overline{z}\xi}.$$
Note that $g$ is an isometry, it maps $B_{\mathbb{D}}^{c}(z, r)$ onto the ball with center zero and radius $r$ in the unit disc. Therefore, by \eqref{eq7:sub}
$$B_{\mathbb{D}}^{c}(0, r) =\lbrace \xi\in \mathbb{D} : c_{\mathbb{D}}(0, \xi) < r \rbrace \subseteq g(\Omega).$$
Thus, we get  
\begin{equation}\label{eq12:sub}
H_{\Omega}^{c}(z) \leq \tanh r \leq \displaystyle\min_{w\in K} \left\lbrace \left|\dfrac{w-z}{1-\overline{z}w}\right|\right\rbrace.
\end{equation}
Hence by {\cite[Theorem~1(1)]{recent2}}, \eqref{eq5:sub} and \eqref{eq12:sub}
$$\min_{w\in K} \left\lbrace \left|\dfrac{w-z}{1-\overline{z}w}\right|\right\rbrace \leq i_{\Omega}^{c}(z) \leq H_{\Omega}^{c}(z) \leq \displaystyle\min_{w\in K} \left\lbrace \left|\dfrac{w-z}{1-\overline{z}w}\right|\right\rbrace.$$
\end{proof}

\begin{proof}[Proof of Theorem \ref{main10}]
For $z\in D$, let $H_{D}^{c}(z) = c$. Then by definition there exist a sequence of positive numbers $\lbrace r_{n} \rbrace$ converging to $\tanh^{-1}c$ and a sequence of injective holomorphic functions $f_{n} : \mathbb{D}\to D$ such that $B_{D}^{c}(z, r_{n}) \subseteq f_{n}(\mathbb{D}) \subseteq D$ for all $n$. It is easy to see that $c_{D}(\xi_{1}, \xi_{2}) = c_{\Omega}(\xi_{1}, \xi_{2})$ for all $\xi_{1}, \xi_{2}\in D$. This implies
\begin{equation}\label{eq22:sub}
B_{D}^{c}(z, r_{n}) = \lbrace \xi\in D : c_{\Omega}(z, \xi) < r_{n}\rbrace \subseteq f_{n}(\mathbb{D}) \subseteq D\ \mbox{for all}\ n.
\end{equation}
Since $f_{n}(\mathbb{D})$ is simply connected domain in $D$, we have $B_{D}^{c}(z, r_{n}) = B_{\Omega}^{c}(z, r_{n})$ for all $n$. By \eqref{eq22:sub} 
\begin{equation}\label{eq20:sub}
B_{\Omega}^{c}(z, r_{n}) = \lbrace \xi\in \Omega : c_{\Omega}(z, \xi) < r_{n}\rbrace\subseteq f_{n}(\mathbb{D}) \subseteq D\ \mbox{for all}\ n.
\end{equation}
If $c > \displaystyle\min_{w\in K} \tanh[c_{\Omega}(z, w)] $, then there is some $w'\in K$ such that $w' \in B_{\Omega}^{c}(z, r_{n})$ for some large $n$. This is not possible because $B_{\Omega}^{c}(z, r_{n})\subseteq D$ for all $n$. Thus
\begin{equation}\label{eq21:sub}
H_{\Omega}^{c}(z) \leq \min_{w\in K} \tanh[c_{\Omega}(z, w)].
\end{equation}
Hence the result follows from \cite[Theorem~1(1)]{recent2} and \eqref{eq21:sub}.
\end{proof}

\begin{proof}[Proof of Theorem \ref{main11}]
For $z\in (r, 1)$, let us assume $i_{D}^{c}(z) = \tanh[c_{A_{r}}(z, p)]$. Then by definition, there exists a sequence of positive numbers $\lbrace r_{n} \rbrace$ converging to $c_{A_{r}}(z, p)$ such that $B_{D}^{c}(z, r_{n})$ are simply connected for all $n$. By Lemma \ref{main200}, $c_{A_{r}}(z, -\sqrt{r}) < c_{A_{r}}(z, \xi)$ for all $\xi \in (-1, -r)\setminus \lbrace -\sqrt{r} \rbrace$, this implies that $-\sqrt{r}\in B_{D}^{c}(z, r_{n})$ for large $n$. Observe that $B_{D}^{c}(z, r_{n}) = B_{A_{r}}^{c}(z, r_{n})$ for large $n$. By \cite{Frer}, $B_{A_{r}}^{c}(z, r_{n})$ is connected. Hence $B_{A_{r}}^{c}(z, r_{n})$ is path connected in $\mathbb{C}$. Also, by Lemma \ref{main500}, $B_{A_{r}}^{c}(z, r_{n})$ is symmetric about $x$- axis. This implies there exist two paths $\gamma_{1}$ and $\gamma_{2}$ in $B_{A_{r}}^{c}(z, r_{n})$ connecting $z$ and $-\sqrt{r}$, lying on the upper half plane and lower half plane respectively for large $n$. Let $\gamma$ denote the closed curve in $B_{A_{r}}^{c}(z, r_{n})$ induced by $\gamma_{1}$ and $\gamma_{2}$ for large $n$. Then $\gamma$ is a closed curve in $A_{r}$. Note that $\gamma$ is not null-homotopic in $A_{r}$. This is a contradiction because $B_{D}^{c}(z, r_{n}) = B_{A_{r}}^{c}(z, r_{n})$ is simply connected. 

For $z\in (r, 1)$, let us assume $H_{D}^{c}(z) = \tanh[c_{A_{r}}(z, p)]$. Then by definition, there exist a sequence of positive numbers $\lbrace r_{n} \rbrace$ converging to $c_{A_{r}}(z, p)$ and a sequence of injective holomorphic functions $f_{n} : \mathbb{D}\to D$ such that $B_{D}^{c}(z, r_{n}) \subseteq f_{n}(\mathbb{D}) \subseteq D$ for all $n$. Observe that $B_{D}^{c}(z, r_{n}) = B_{A_{r}}^{c}(z, r_{n})$ for large $n$. Similarly as above, we can find a closed curve $\gamma$ in $B_{A_{r}}^{c}(z, r_{n})$ which is not null-homotopic in $A_{r}$ for large $n$. This contradict the fact that $f_{n}(\mathbb{D})$ is simply connected for all $n$.

\end{proof}

\begin{proof}[Proof of Theorem \ref{main2}]
For $z\in \Omega $, by using the automorphism of the unit disc, it is easy to see that
\begin{equation}\label{eq1:sub}
S_{\Omega}(z) \geq \min_{w\in K} \left\lbrace \left|\dfrac{w-z}{1-\overline{z}w}\right|\right\rbrace.
\end{equation}
By {\cite[Theorem~2.1]{2012}}, there is an injective holomorphic function $f :\Omega \to \mathbb{D}$ with $f(z) = 0$ such that $\mathbb{D}(0 , S_{\Omega}(z))\subseteq f(\Omega)$. Since $K$ is removable set, there exists a holomorphic function $\tilde{f}:\mathbb{D}\to \mathbb{C}$ with $\tilde{f}(\xi) = f(\xi)$ for all $\xi\in \Omega$. Note that $K$ have no interior. This implies $\tilde{f}:\mathbb{D}\to \overline{\mathbb{D}}$ and by the maximum modulus principle $\tilde{f}:\mathbb{D}\to \mathbb{D}$. It is easy to see that $\tilde{f}(w)\notin f(\Omega)$ for all $w\in K$. This implies
\begin{equation}\label{eq2:sub}
S_{\Omega}(z) \leq |\tilde{f}(w)|\ \mbox{for all}\ w\in K.
\end{equation}
By distance decreasing property of Poincaré metric
\begin{align*}
d_{\rho}(0 , \tilde{f}(w)) = d_{\rho}(\tilde{f}(z) , \tilde{f}(w)) \leq d_{\rho}(z , w) = d_{\rho}\left(0 , \dfrac{w-z}{1-\overline{z}w}\right).
\end{align*}
Therefore,
\begin{align*}
\mu(|\tilde{f}(w)|)&\leq \mu\left(\left|\dfrac{w-z}{1-\overline{z}w}\right|\right).
\end{align*}
Thus, we get
\begin{equation}\label{eq3:sub}
|\tilde{f}(w)|\leq \left|\dfrac{w-z}{1-\overline{z}w}\right|\ \mbox{for all}\ w\in K.
\end{equation}
By \eqref{eq2:sub} and \eqref{eq3:sub},
\begin{equation}\label{eq4:sub}
S_{\Omega}(z) \leq \min_{w\in K} \left\lbrace \left|\dfrac{w-z}{1-\overline{z}w}\right|\right\rbrace.
\end{equation}
Hence the result follows from \eqref{eq1:sub} and \eqref{eq4:sub}.
\end{proof}

\begin{proof}[Proof of Theorem \ref{main3}]
For $z\in \Omega$, by {\cite[Theorem~2.1]{2012}}, there exists an injective holomorphic function $f :D\to \mathbb{D}$ with $f(z) = 0$ such that $\mathbb{D}(0 , S_{D}(z))\subseteq f(D)$. Similar to the Proof of Theorem \ref{main2}, we can find a holomorphic function $\tilde{f}: \Omega\to \mathbb{D}$ with $\tilde{f}(\xi) = f(\xi)$ for all $\xi\in D$ and $\tilde{f}(w)\notin f(D)$ for all $w\in K$. By decreasing property of Carathéodory metric 
\begin{equation}
c_{\mathbb{D}}(\tilde{f}(z), \tilde{f}(w)) \leq c_{\Omega}(z, w)\ \mbox{for all}\ w\in K.
\end{equation}
\begin{align*}
c_{\mathbb{D}}(0, \tilde{f}(w)) = \tanh^{-1}|\tilde{f}(w)| \leq c_{\Omega}(z, w)\ \mbox{for all}\ w\in K.
\end{align*}
Therefore,
\begin{align*}
S_{\Omega}(z) \leq |\tilde{f}(w)|  \leq \tanh[c_{\Omega}(z, w)]\ \mbox{for all}\ w\in K.
\end{align*}
Hence
\begin{align*}
S_{D}(z) \leq \min_{w\in K} \tanh[c_{\Omega}(z, w)].
\end{align*}
\end{proof}

\begin{proof}[Proof of Theorem \ref{main15}]

For $z\in(r, \sqrt{r})$, suppose that $S_{D}(z) = \tanh [c_{A_{r}}(z, -\sqrt{r})]$.
This implies there exists a injective holomorphic function $f : D \to \mathbb{D}$ with $f(z) = 0$ such that $\mathbb{D}(0 , \tanh c_{A_{r}}(z, -\sqrt{r}))\subseteq f(D)$. Therefore, there is an injective holomorphic function $\tilde{f}: A_{r}\to \mathbb{D}$ with $\tilde{f}(\xi) = f(\xi)$ for all $\xi\in D$. It follows that 
\begin{align}\label{eq201:sub}
S_{A_{r}}(z) \geq \tanh c_{A_{r}}(z, -\sqrt{r}).
\end{align}
By \eqref{eq401:sub} and \eqref{eq501:sub}
\begin{align}\label{eq202:sub}
\tanh c_{A_{r}}(z, -\sqrt{r}) = \dfrac{1}{z}\left(\dfrac{\omega(z, -r^{1/2})}{\sqrt{r}\omega(z,-r^{-1/2})}\right)^{2}.
\end{align}
Consider 
$$ f(z) = \dfrac{\omega(z, -r^{1/2})}{\sqrt{r}\omega(z,-r^{-1/2})}$$
for $z\in A_{r}$. This is a conformal map from $A_{r}$ onto circularly slit disc of radius $\sqrt{r}$ with $f(-\sqrt{r}) = 0$ and $f(\partial\mathbb{D})= \partial\mathbb{D}$. By \eqref{eq1002:sub}, it is easy to see that $f(\overline{z})= \overline{f(z)}$. It follows that $f(z)\in (-1, \sqrt{r})$ for $z\in(-1, -r)$ and $f(z)\in (\sqrt{r}, 1)$ for $z\in(r, 1)$. Therefore, by \eqref{eq202:sub}
$$\tanh c_{A_{r}}(z, -\sqrt{r}) > \dfrac{r}{z}.$$
Hence, by \eqref{eq201:sub}
\begin{align*}
S_{A_{r}}(z) > \dfrac{r}{z},
\end{align*}
which is not possible by the expression of squeezing function on annulus. In \cite{recent}, the expression of squeezing function on annulus is given by 
$$ S_{A_{r}}(z) = \max \left\lbrace |z|, \dfrac{r}{|z|} \right\rbrace.$$
\end{proof}

\section*{Acknowledgement}
We are thankful to Prof. Kaushal Verma for raising the problems related to the removable sets.

\end{document}